\documentclass[10pt,reqno]{amsart}
\usepackage[10pt,nona4]{optional}
\usepackage{dsfont,amsmath,amsfonts,amscd,amssymb,graphicx,mathrsfs,eufrak,upgreek,pdflscape,colonequals}
\usepackage[usenames]{color}
\usepackage{setspace}
\usepackage{array,multirow,booktabs,longtable}

\usepackage[colorlinks]{hyperref}
\usepackage{tikz,mathrsfs}
\usetikzlibrary{arrows,decorations.pathmorphing,decorations.pathreplacing,positioning,shapes.geometric,shapes.misc,decorations.markings,decorations.fractals,calc,patterns}

\tikzset{
        cvertex/.style={circle,draw=black,inner sep=1pt,outer sep=3pt},
        vertex/.style={circle,fill=black,inner sep=1pt,outer sep=3pt},
        DB/.style={circle,draw=black,circle,fill=black,inner sep=0pt, minimum size=4pt},
        DW/.style={circle,draw=black,inner sep=0pt, minimum size=4pt},
        star/.style={circle,fill=yellow,inner sep=0.75pt,outer sep=0.75pt},
        tvertex/.style={inner sep=1pt,font=\scriptsize},
        gap/.style={inner sep=0.5pt,fill=white}}

\newtheorem{thm}{Theorem}[section]
\newtheorem{prop}[thm]{Proposition}
\newtheorem{lemma}[thm]{Lemma}

\newtheorem{cor}[thm]{Corollary}

\theoremstyle{definition} 

\newtheorem{example}[thm]{Example}

\newtheorem{remark}[thm]{Remark}

\numberwithin{equation}{section}

\newcommand{\n}{\mathfrak{n}}

\renewcommand{\t}[1]{\textnormal{#1}}

\def\CM{\mathop{\rm CM}\nolimits}

\def\mod{\mathop{\rm mod}\nolimits}

\def\Qcoh{\mathop{\rm Qcoh}\nolimits}

\def\uEnd{\mathop{\underline{\rm End}}\nolimits}

\def\End{\mathop{\rm End}\nolimits}
\def\Ext{\mathop{\rm Ext}\nolimits}

\def\Spec{\mathop{\rm Spec}\nolimits}

\newcommand{\con}{\mathrm{con}}

\newcommand{\CA}{\mathrm{A}_{\con}}

\def\ab{\mathop{\rm ab}\nolimits}
\def\redu{\mathop{\rm red}\nolimits}

\newcommand\art{\mathsf{Art}}

\newcommand\Sets{\mathsf{Sets}}

\newcommand\Def{\mathcal{D}ef}

\newcommand{\cF}{\mathcal{F}}

\newcommand{\cL}{\mathcal{L}}

\newcommand{\cO}{\mathcal{O}}

\newcommand{\cX}{\mathcal{X}}
\newcommand{\cY}{\mathcal{Y}}

\newcommand{\dmapsto}{\rotatebox[origin=c]{-90}{$\mapsto$}}

\newlength\tempWidth

\newcommand{\defColor}{gray!50}

\def\gridShift{-4}
\def\bend{0.1}
\def\len{0.8}
\def\dotsize{0.1}

\newcommand{\defCurveCoords}[3]{
  (#1+\bend,-#3,#2) .. controls (#1-\bend,-#3/3,#2) and (#1-\bend, #3/3,#2) .. (#1+\bend,#3,#2)
}

\newcommand{\defDrawDot}[2]{
    \draw[fill] (#1+\dotsize,\gridShift,#2+\dotsize) -- (#1+\dotsize,\gridShift,#2-\dotsize) -- (#1-\dotsize,\gridShift,#2-\dotsize) -- (#1-\dotsize,\gridShift,#2+\dotsize) -- cycle
    }

\newcommand{\defCurve}[3]{
  \draw[color=#3] \defCurveCoords{#1}{#2}{\len}
}

\newcommand{\defUniversalCurves}[4]{
  \foreach \i in {0,...,#4} {
    \ifx\i\zero \defCurve{#1+\i}{#2}{black} \else \defCurve{#1+\i}{#2}{\defColor} \fi;
    \defDrawDot{#1+\i}{#2};
  };
  \foreach \j in {1,...,#3} {
    \defCurve{#1}{#2+\j}{\defColor};
    \defDrawDot{#1}{#2+\j};
  }
}

\begin{document}
\title{\textsc{Gopakumar--Vafa invariants do not determine flops}}
\author{Gavin Brown}
\address{Gavin Brown, Mathematics Institute, Zeeman Building, University of Warwick, Coventry, CV4 7AL, UK.}
\email{G.Brown@warwick.ac.uk}
\author{Michael Wemyss}
\address{Michael Wemyss, The Mathematics and Statistics Building,
University of Glasgow, University Place, Glasgow, G12 8SQ, UK.}
\email{michael.wemyss@glasgow.ac.uk}
\begin{abstract}
Two $3$-fold flops are exhibited, both of which have precisely one flopping curve.  One of the two flops is new, and is distinct from all known algebraic $D_4$-flops.   It is shown that the two flops are neither algebraically nor analytically isomorphic, yet their curve-counting Gopakumar--Vafa invariants are the same.  We further show that the contraction algebras associated to both are not isomorphic, so the flops are distinguished at this level.  This shows that the contraction algebra is a finer invariant than various curve-counting theories, and it also provides more evidence for the proposed analytic classification of $3$-fold flops via contraction algebras.
\end{abstract}
\subjclass[2010]{Primary 14E30; Secondary 14J30, 16S38}
\thanks{M.W.\ was supported by EPSRC grant~EP/K021400/2.}
\maketitle
\parindent 20pt
\parskip 0pt

\section{Introduction}

Flopping neighbourhoods are one of the most elementary building blocks of higher dimensional algebraic geometry, and even in dimension three they exhibit a very rich structure. Over the past thirty years the invariants attached to such curves have become increasingly fine, from the trichotomy in the normal bundle \cite{Laufer}, to the \emph{length} in \cite{Kollar}, to the ADE identification in \cite{KM,Kawa}, to the association of a finite tuple of integers via  the Gopakumar--Vafa (=GV) invariants \cite{KatzGV}.    At each stage, the produced invariant is strictly finer than the last, with the GV invariants linking to Donaldson--Thomas theory and all other modern curve counting notions (see e.g.\ \cite{PT}). 

On the other hand, contraction algebras were introduced in \cite{DW1}, partially to provide a new curve invariant, but mainly to unify the homological approaches to derived symmetries and twists \cite{Bridgeland, Chen, Toda}.   With their roots in homological algebra, and because they are an algebra as opposed to a number, this additional structure allows us to use contraction algebras to establish and control many geometric processes \cite{DW2,HomMMP}, whilst at the same time recover the GV and other invariants \cite{DW1,TodaWidth, HuaToda} in a variety of natural ways.

In this paper, we use the algebra structure to show that the contraction algebra is a strictly finer invariant than that of Gopakumar--Vafa.  This is in some ways surprising: the GV invariants are indeed enough to classify Type $A$ flops \cite{Pagoda}.  The trick is to use noncommutativity.  We produce two flops, and we show that their contraction algebras are not isomorphic, although both have the same dimension.  Aside from the issue of actually constructing such an example, which we come back to below, we remark here that the isomorphism problem is delicate and is in general also difficult.  Deciding when two finite dimensional algebras are not isomorphic is tricky, especially in the situation here, when by design all the standard numerical information attached to each is the same.

The main result is the following, where $g$ is the standard Laufer flop \cite{Laufer}.

\begin{thm}[\ref{not iso}, \ref{GV same}, \ref{Acon not iso}]\label{main intro}
Consider the flopping contractions $f\colon X\to\Spec R$ and $g\colon Y\to\Spec L$ constructed in \ref{ex1} and \ref{ex2}. Then the following statements hold.
\begin{enumerate}
\item $R$ is not analytically (or algebraically) isomorphic to $L$.
\item The Gopakumar--Vafa invariants associated to $f$ and $g$ are the same.
\item The contraction algebras associated to $f$ and $g$ are not isomorphic.
\end{enumerate}
\end{thm}

It is conjectured in \cite[1.4]{DW1} that contraction algebras are  \emph{the} analytic classification of $3$-fold flops.  Whilst the new flop $f$ may look like it comes out of the blue, we found it during our systematic attempts to approach the conjecture based on an explicit gluing via a superpotential; on the noncommutative side, the example is much clearer.  Indeed, the flop $f$ was constructed by \emph{assuming} the above conjecture is true, and working backwards, thus the results in this paper add some weight to the conjecture.  We also remark that there are  tables of data which numerically suggest, but do not quite yet prove, that different flops having the same GV invariants is actually quite typical behaviour. 

It is perhaps worth explaining the heuristic reason as to why the noncommutativity of the contraction algebra helps, rather than hinders, distinguishing the two flops above.  Whilst algebraically the two commutative curves $x^3-y^2$ and $x^3(x+1)-y^2$ are different, analytically around the origin, their coordinate rings are isomorphic.  Set $\widehat{\mathbb{A}}_{\uplambda}\colonequals \mathbb{C}\langle\!\langle x,y\rangle\!\rangle/ xy-\uplambda yx$, where implicitly we consider the closure of all ideals.  Then the above famous algebro-geometric curve example is precisely the statement that
\begin{equation}
\frac{\widehat{\mathbb{A}}_{1}}{x^3-y^2}\cong
\frac{\widehat{\mathbb{A}}_1}{x^3(x+1)-y^2}.\label{2 cont alg A}
\end{equation}
The third part of \ref{main intro} turns out to be equivalent to establishing the more surprising statement that there is no such isomorphism in the quantum plane, namely
\begin{equation}
\frac{\widehat{\mathbb{A}}_{-1}}{x^3-y^2}\ncong
\frac{\widehat{\mathbb{A}}_{-1}}{x^3(x+1)-y^2}.\label{2 cont alg}
\end{equation}
Heuristically, noncommutativity gives the flexibility to distinguish: since  $y$ no longer commutes with $x$, it no longer commutes with $\sqrt{x+1}$, so we should expect the isomorphism in \eqref{2 cont alg A} to break down.  It turns out that the left hand side of \eqref{2 cont alg} is the contraction algebra of $g$ \cite{DW1}, and we show in \S\ref{calc cont section} and \ref{intro ok} that the right hand side of \eqref{2 cont alg} is the contraction algebra of $f$.  The proof of \eqref{2 cont alg} is somewhat more involved than this heuristic argument; we give a direct proof in \ref{Acon not iso}, but it is also possible to give a computer algebra verification by adapting the Shirayanagi algorithm \cite{Shirayanagi}.

\subsection*{Acknowledgements} The authors would like to thank Agata Smoktunowicz and Natalia Iyudu for many helpful discussions, and for sharing with us a direct proof that the algebra in \ref{L cont alg} and the algebra in \ref{intro ok} are not isomorphic.

\section{The Two Flops}

Here the two main examples are introduced.  All the calculations in \ref{ex1} are easy enough to be done by hand, but to allay any possibilities of error, we have included in Appendix~\ref{theappendix} computer algebra codings that can be used to independently check all claims.

\begin{example}[The new flop $R$]\label{ex1}
Consider the hypersurface $R\colonequals \mathbb{C}[u,v,x,y]/(f)$, where 
\[
f\colonequals u^2+v^2(x+y)+x(x^2+xy^2+y^3).
\]
By the Jacobi criterion, $R$ has a unique isolated singular point, at the origin. Being a hypersurface in $\mathbb{A}^4$, clearly $\Spec R$ is a Gorenstein $3$-fold.  We next verify that it is the base of a simple flopping contraction, by constructing a small resolution.  The same calculation shows that $R$ is $cD_4$, although this can also be verified at once from the above equation.

To construct a small resolution, blowup the reflexive ideal $I\colonequals (vx-uy, xy^2+v^2, x^2y+uv)$ to obtain a projective birational morphism
\[
X\to \Spec R.
\]
Here we summarise the calculation by hand; the computation using Singular is summarised in the appendix \S\ref{app flops}. The blowup $X$ is covered by  two affine open charts, the first of which is given by the smooth hypersurface
\[
U_1\colonequals \Spec \mathbb{C}[x_3,x_4,y_1,y_2]/(x_3(y_1^2+1)+x_4y_1^2+y_2^2)
\]
with map to the base
\[
\begin{array}{c}
(x_3,x_4,y_1,y_2)\in U_1\\
\dmapsto\\
(x_3x_4y_1+x_4^2y_1+x_3y_2,x_3y_1-x_4y_2,x_3,x_4)\in \Spec R.
\end{array}
\]
Above the origin of $\Spec R$ consists of all points $(0,0,y_1,y_2)$ of $U_1$ such that the defining relation of $U_1$ holds, so necessarily $y_2^2=0$. Thus the fibre above the origin is a single curve, with scheme multiplicity two.  The second open chart is given by the smooth hypersurface
\[
U_2\colonequals \Spec \mathbb{C}[x_2, x_4, y_0, y_2]/(x_2y_0^3+x_4y_0^2y_2+x_2y_0+x_4y_2+y_2^2+x_4)
\]
with map to the base
\[
\begin{array}{c}
(x_2,x_4,y_0,y_2)\in U_2\\
\dmapsto\\
(-x_2x_4y_0^2-x_4^2y_0y_2+x_2y_2, x_2,x_2y_0+x_4y_2,x_4)\in \Spec R.
\end{array}
\]
Here the fibre above the origin consists of $(0,0,y_0,y_2)$ such that $y_2^2=0$, which again is a curve.  It is an easy check to see that the reduced fibre above the origin glues via
\[
(0,0,y_1,0)\leftrightarrow (0,0,y_1^{-1},0)
\]
and so is $\mathbb{P}^1$.  It follows that $X\to\Spec R$ is a smooth flopping contraction, and thus $R$ is cDV.  Since by the above calculation the scheme fibre has multiplicity two, we deduce that this must be a $cD_4$ flop \cite{KM,Kawa}.  
\end{example}

\begin{example}[The standard Laufer flop $L$]\label{ex2}
Consider $L\colonequals \mathbb{C}[u,v,x,y]/(g)$, where
\[
g\colonequals u^2+v^2y-x(x^2+y^3).
\]
This has a unique singular point at the origin, and indeed $\Spec L$ is the base of the standard Laufer flop.  Blowing up the reflexive ideal $(x^2+y^3,vx+uy,ux-vy^2)$ gives a projective birational morphism
\[
Y\to\Spec L
\]
where $Y$ is smooth.  The reduced scheme fibre above the origin is $\mathbb{P}^1$, and the full scheme fibre has multiplicity two.  This was the first known example of a $cD_4$ flop \cite{Laufer,KM,Pagoda}.
\end{example}

\begin{remark}\label{not iso}
$R$ is not analytically isomorphic to $L$, and hence also $R\ncong L$  algebraically.  This can be seen directly by computing the Tjurina numbers of both (see e.g.\ \S\ref{app flops}), but it also follows from the non-isomorphism of the contraction algebras later in \ref{Acon not iso}.
\end{remark}

\section{GV invariants and Contraction algebras}

The GV invariants of both the flopping contractions $X\to\Spec R$ and $Y\to \Spec L$ in the previous section are determined by their contraction algebras \cite{TodaWidth}, and this section briefly reviews these notions.  

\subsection{Contraction Algebra background}
Throughout, consider a general $3$-fold flopping contraction $f\colon U\to\Spec \mathfrak{R}$, where $U$ is smooth, $f^{-1}(0)\colonequals C$ and $C^{\redu}\cong \mathbb{P}^1$, and for simplicity assume that $\mathfrak{R}$ is complete local.  To this data, one can associate the \emph{contraction algebra} $\CA$, which can be defined \cite[\S3]{DW1} as the representing object of the noncommutative deformation functor
\[
\Def\colon\art_1\to\Sets,
\]
where $\art_1$ is the category of augmented finite dimensional $\mathbb{C}$-algebras. By definition $\Def$ sends
\[
(\Gamma,\n)\mapsto \left. \left \{ (\cF,\upphi,\updelta)
\left|\begin{array}{l}\cF\in\Qcoh U\\
\upphi\colon \Gamma\to\End_U(\cF) \mbox{ is a $\mathbb{C}$-algebra homomorphism}\\ 
-\otimes_\Gamma \cF\colon\mod \Gamma\to \Qcoh U \t{ is exact}\\ \updelta\colon (\Gamma/\n)\otimes_\Gamma \cF\xrightarrow{\sim}\cO_{\mathbb{P}^1}(-1)\end{array}\right. \right\} \middle/ \sim \right.
\]
where the equivalence relation $\sim$ is outlined in detail in \cite[2.4]{DW1}.  For the purpose of this paper, the following information suffices.
\begin{thm}\label{CA summary}
Consider a complete local flopping contraction $U\to\Spec\mathfrak{R}$, as above.  Then the following statements hold.
\begin{enumerate}
\item\label{CA summary 1} $\CA$ is a finite dimensional algebra.
\item\label{CA summary 2} $\CA$ is not commutative if and only if $C$ is a $(-3,1)$-curve.  Moreover, in this case $\CA$ can be presented as
\[
\CA\cong\frac{\mathbb{C}\langle\!\langle x,y\rangle\!\rangle}{(\delta_xW, \delta_yW)} 
\]
for some superpotential $W$, where $(\delta_xW, \delta_yW)$ denotes the closure of the ideal generated by the formal derivatives $\delta_xW$ and $\delta_yW$.
\item\label{CA summary 3} Suppose that $C$ is a $(-3,1)$-curve.  If $X\in\CM \mathfrak{R}$ is a non-free rank two module, and $\Ext^1_{\mathfrak{R}}(X,X)=0$, then $\CA\cong\uEnd_{\mathfrak{R}}(X)$ and $\dim_{\mathbb{C}}\CA=\dim_{\mathbb{C}}\Ext^2_{\mathfrak{R}}(X,X)$.
\end{enumerate}
\end{thm}
\begin{proof}
Part \eqref{CA summary 1} is \cite[2.13(1)]{DW1}, and the first statement of part \eqref{CA summary 2} is \cite[2.13]{DW1}. The fact that $\CA$ is a superpotential algebra is a consequence of that it is a factor of an NCCR \cite[\S3]{DW1}, which since $\mathfrak{R}$ is complete, is a superpotential algebra \cite{VdBCY}.

For part \eqref{CA summary 3}, by \cite[4.13]{HomMMP} there are only two non-free indecomposable CM $\mathfrak{R}$-modules $M_i$ for which $\Ext^1_{\mathfrak{R}}(M_i,M_i)=0$, and the rank of each $M_i$ equals the length of the flopping curve.  Since the curve is a $(-3,1)$-curve, the length is strictly greater than one.  Hence $X$ cannot split into two rank one summands, as then there would be a rank one rigid non-free CM $\mathfrak{R}$-module. We conclude that $X$ is indecomposable, so it must be isomorphic to one of the $M_i$.  By definition, $\CA$ is $\uEnd_{\mathfrak{R}}(M_i)$, and hence is isomorphic to $\uEnd_{\mathfrak{R}}(X)$.  The last statement regarding the dimension is then \cite[5.2]{DW1}.
\end{proof}

Later, \ref{CA summary}\eqref{CA summary 3} will be used to calculate the contraction algebra, and also to compute its dimension, without requiring  knowledge of its algebra structure.

\subsection{GV Invariants} 
Each flopping contraction $f$ of length $\ell$ has an associated tuple of integers $(n_1,\hdots,n_\ell)$ called the \emph{Gopakumar--Vafa} invariants.  These can be defined as follows. As in \cite[\S2.1]{BKL}, there exists a flat deformation
\[
\begin{tikzpicture}[>=stealth]
\node (X) at (0,0) {$\cX$};
\node (Y) at (0,-1) {$\cY$};
\node (T) at (0.75,-1.75) {$T$};
\draw[->] (X)--(Y);
\draw[->] (Y)--(T);
\draw[->,densely dotted] (X)--(T);
\end{tikzpicture}
\]
for some Zariski open neighbourhood $T$ of $0\in\mathbb{A}^1$, such that
\begin{itemize}
\item The central fibre $g_0\colon X_0\to Y_0$ is isomorphic to the formal fibre $\widehat{f}$ of $f$.
\item All other fibres $g_t\colon X_t\to Y_t$ for $t\in T\backslash \{0\}$  are flopping contractions whose exceptional locus is a disjoint union of $(-1,-1)$-curves.
\end{itemize}
Regarding the flopping curve $C$ of $\widehat{f}$ as a curve in the central fibre of $\mathcal{X}\to T$, and thus as a curve in $\mathcal{X}$, then the GV invariant $n_j$ is defined to be the number of $g_t$-exceptional $(-1,-1)$-curves $C'$ with curve class $j[C]$, i.e.\ for every line bundle $\cL$ on $\cX$, 
\[
\deg (\cL|_{C'})=j \deg(\cL|_C).
\] 
The following is \cite[1.1]{TodaWidth}, and will be used to deduce the GV invariants later.

\begin{thm}[Toda]\label{num prop}
Suppose that $f\colon U\to \Spec \mathfrak{R}$ is a complete local flopping contraction of a single length $\ell$ $(-3,1)$-curve, where $U$ is smooth.  Then $n_1=\CA^{\ab}$ and
 \[
 \dim_{\mathbb{C}}\CA=\dim_{\mathbb{C}}\CA^{\ab}+\sum_{j=2}^{\ell}j^2\cdot n_j,
 \]
 where $n_j\in\mathbb{Z}_{\geq 1}$ are the Gopakumar--Vafa invariants associated to the curve.
\end{thm}

\section{GV invariants do not determine flops}\label{calculate GV section}

This section computes the contraction algebras for the flopping contractions \ref{ex1} and \ref{ex2}, and as a corollary shows that the GV invariants attached to both flops are the same.  The two contraction algebras are then shown not to be isomorphic, and so the flops are distinguished at this finer level.

\subsection{Calculation of Contraction Algebras}\label{calc cont section}
Write $\Lambda_{\con}$ for the contraction algebra associated to the standard Laufer flop $Y\to\Spec L$ in \ref{ex2}.  The following is known.

\begin{example}\label{L cont alg}
With notation as in the introduction, 
\[
\Lambda_{\con}\cong\widehat{\Lambda}_{\con}\cong\frac{\mathbb{C}\langle x,y\rangle}{(xy+yx,x^3-y^2)}=\frac{\mathbb{A}_{-1}}{x^3-y^2},
\]
where the first isomorphism is \cite[2.17]{DW1} and the second is \cite[1.3]{DW1}.  Thus $\Lambda_{\con}$ is a $9$-dimensional not-commutative ring, given by superpotential $W=x^4-xy^2$.
\end{example}

The calculation of the contraction algebra associated to the new flop uses a very similar method to the above example.  To set notation,  write $\Gamma_{\con}$ for the contraction algebra associated to $X\to \Spec R$ in \ref{ex1}.

\begin{prop}\label{R cont alg}
There is an isomorphism
\[
\Gamma_{\con}\cong\frac{\mathbb{C}\langle a,b\rangle}{ab+ba, -a^2 + b^3 + aba}.
\]
This is a nine-dimensional not-commutative ring.
\end{prop}
\begin{proof}
The most direct method to prove this is to specialise the universal flop
\[
u^2+v^2Y+x^2\upalpha + 2\upbeta xv + (\upalpha Y-\upbeta^2)\upgamma^2
\]
given in \cite[(46)]{AM} at $\upalpha=x$, $Y=x+y$, $\upbeta=0$ and $\upgamma=y$ to give
\[
u^2+v^2(x+y)+x^3+x(x+y)y^2,
\]
which equals $f$ in \ref{ex1}.  Consider the cokernel $M$ of the matrix $\Upphi$ in the following matrix factorisation
\[
R^4\xrightarrow{\Uppsi\colonequals\left(\begin{smallmatrix} u&x&v&y\\
-x^2&u&-xy&v\\
-vx-vy&xy+y^2&u&-x\\
-x^2y-xy^2&-vx-vy&x^2&u\end{smallmatrix}\right)} 
R^4\xrightarrow{\Upphi\colonequals\left(\begin{smallmatrix} 
u&-x&-v&-y\\
x^2&u&xy&-v\\
vx+vy&-(xy+y^2)&u&x\\
x^2y+xy^2&vx+vy&-x^2&u\end{smallmatrix}\right)} R^4.
\]
Then by \cite[\S 4]{AM} $\End_R(R\oplus M)$ is presented as the path algebra of the quiver 
\[
\begin{array}{c}
\begin{tikzpicture}[>=stealth]
\node (C1) at (0,0) {$\scriptstyle R$};
\node (C1a) at (-0.1,0)  {};
\node (C2) at (1.75,0) {$\scriptstyle M$};
\node (C2a) at (1.85,0.05) {};
\node (C2b) at (1.85,-0.05) {};
\draw [->,bend left=20,looseness=1,pos=0.5] (C1) to node[above]  {$\scriptstyle d=\left(\begin{smallmatrix} 0\\0\\0\\1\end{smallmatrix}\right)$} (C2);
\draw [->,bend left=20,looseness=1,pos=0.5] (C2) to node[below]  {$\scriptstyle c=(\begin{smallmatrix} u&-x&-v&-y\end{smallmatrix})\quad$} (C1);
\draw[->]  (C2b) edge [in=-80,out=-10,loop,looseness=12,pos=0.5] node[below right, pos=0.05] {$\scriptstyle b=\left(\begin{smallmatrix} 0&0&1&0\\0&0&0&-1\\-(x+y)&0&0&0\\0&x+y&0&0\end{smallmatrix}\right)$} (C2b);
\draw[->]  (C2a) edge [in=80,out=10,loop,looseness=12,pos=0.5] node[above right=-0.25] 
{$\scriptstyle a=\left(\begin{smallmatrix} 0&1&0&0\\-x&0&0&0\\0&0&0&1\\0&0&-x&0\end{smallmatrix}\right)$} (C2a);
\end{tikzpicture}
\end{array}
\] 
where the relations are determined by matrix multiplication, up to the column space of the matrix $\Uppsi$.  For example, using the above matrices, it can be seen directly that $a\circ b+b\circ a=0$, and that
\[
-a^2+b^3+a\circ b\circ a - b\circ d\circ c - d\circ c\circ b=\left(\begin{smallmatrix} x&0&-y&0\\
u&0&-v&0\\
xy+y^2&0&x&0\\
-vx-vy&0&-u&0\end{smallmatrix}\right),
\]
which belongs to the column space of $\Uppsi$, and thus is zero.  Factoring out the vertex corresponding to $R$ in the above presentation, and killing all arrows that factor through it, it follows that $\Gamma_{\con}=\uEnd_R(M)$  can be presented as $\mathbb{C}\langle a,b\rangle$ subject to at least the relations $ab+ba=0$ and $-a^2+b^3+aba=0$.  Since these relations give an algebra which is nine-dimensional (either by the Diamond Lemma, or  \S\ref{app Acon}), and $\Gamma_{\con}$ is nine-dimensional (see \S\ref{app Acon}), it follows that these are all the relations that are needed. 
\end{proof}

\begin{remark}
It is not obvious from the presentation in \ref{R cont alg}, but it is nevertheless true (see \ref{intro ok}) that $\Gamma_{\con}$ is a superpotential algebra.  
\end{remark}

The following is the main result of this subsection.
\begin{cor}\label{GV same}
The GV invariants attached to the flops in \ref{ex1} and \ref{ex2} are the same.
\end{cor}
\begin{proof}
Via the explicit presentations in \ref{L cont alg} and \ref{R cont alg}, we see immediately using Toda's formula \ref{num prop} that $n_1=5$ and $n_2=1$ in both cases.
\end{proof}

\begin{remark}
It is not necessary to compute the algebra structures in order to deduce that the GV invariants of the two flops are the same.  Using 
\ref{CA summary}\eqref{CA summary 3}, the dimension of both $\Lambda_{\con}$ and $\Gamma_{\con}$ can be seen to be nine directly, without knowing the algebra structure.  We outline the code in appendix \S\ref{app Acon}.  Then, since both are $cD_4$ flops, neither contraction algebra is commutative by \ref{CA summary}\eqref{CA summary 2}.  The abelianization of any not-commutative contraction algebra must be at least four dimensional, being the factor of $\mathbb{C}[[x,y]]$ by two relations in which each word is quadratic or higher.  Hence by Toda's formula \ref{num prop}, since the only possibility of writing $9$ as a sum of squares is $9=5.1^2+1.2^2$, it follows the GV invariants of both flops must be the same, namely $n_1=5$ and $n_2=1$. These numerics are how this example was discovered.
\end{remark}

\subsection{The contraction algebras are not isomorphic}

The proof that $\Lambda_{\con}$ is not isomorphic to $\Gamma_{\con}$ requires the following preparatory lemma.

\begin{lemma}\label{noniso prep}
With notation as above, the following statements hold.
\begin{enumerate}
\item\label{noniso prep 1} In $\Gamma_{\con}$, we have $a^3=0$ and $b^6=0$.
\item\label{noniso prep 2} $\Lambda_{\con}$ has basis $\{ 1, y, x, y^2, yx, x^2, y^2x, yx^2, y^2x^2 \}$.
\end{enumerate}
\end{lemma}
\begin{proof}
(1) Multiplying the defining equation $a^2=b^3+aba$ by $a$ on the right, and by $a$ on the left, it follows that 
\[
ab^3+a^2ba=a^3=b^3a+aba^2.
\]
Using the defining equation $ab=-ba$ repeatedly,  
\[
ab^3-a^3b=a^3=-ab^3+a^3b
\]
Since the left hand side is the negative of the right hand side, it follows that $a^3=0$.   Then, simply squaring both sides of the defining equation $a^2-aba=b^3$ gives
\[
b^6=(a^2-aba)^2=a^4-a^3ba-aba^3-aba^2ba,
\]  
which is zero, using the relation $ab=-ba$ together with the fact that $a^3=0$.\\
(2) The first method to establish this is just to use the Diamond Lemma directly, and indeed the stated basis is exactly the one used in \cite[3.14]{DW1}.  The second method, using magma, is outlined in the appendix (\S\ref{app Acon}).
\end{proof}

\begin{thm}\label{Acon not iso}
$\Lambda_{\con}$ is not isomorphic to $\Gamma_{\con}$.
\end{thm}
\begin{proof}
 Let $\uppsi\colon \Gamma_{\con}\to\Lambda_{\con}$ be an arbitrary isomorphism; we aim for a contradiction. As $\Lambda_{\con},\Gamma_{\con}\in\art_1$, in both cases their Jacobson radical is their path ideal. Hence under the isomorphism $\uppsi$, necessarily the generators $a$ and $b$ must map to the Jacobson radical, and so by \ref{noniso prep}\eqref{noniso prep 2} we may write
\begin{align*}
\uppsi(a) & = \uplambda_1y+\uplambda_2x+\uplambda_3y^2+\hdots+\uplambda_8y^2x^2\\
\uppsi(b)&= \upmu_1y+\upmu_2x+\upmu_3y^2+\hdots+\upmu_8y^2x^2
\end{align*}
for some scalars $\uplambda_1,\hdots,\upmu_8$. Now by \ref{noniso prep}\eqref{noniso prep 1}, $\uppsi(a)^3=\uppsi(a^3)=\uppsi(0)=0$, thus 
\[
(\uplambda_1y+\uplambda_2x+\uplambda_3y^2+\hdots+\uplambda_8y^2x^2)^3=0
\]
in $\Lambda_{\con}$. Multiplying out the left hand side,
 and using the relations of $\Lambda_{\con}$, we can express the left hand side in terms of  the basis \ref{noniso prep}\eqref{noniso prep 2} of $\Lambda_{\con}$ (see e.g.\ \S\ref{app Acon}). Doing this,
\[
(\uplambda_2^3)y^2 +
(\uplambda_1^2\uplambda_2 + 3\uplambda_2^2\uplambda_5) y^2x +
(\uplambda_1\uplambda_2^2)yx^2 +
3(\uplambda_1^2\uplambda_5 + \uplambda_2^2\uplambda_3 + \uplambda_2\uplambda_5^2)y^2x^2=0.
\]
 Being a basis, all coefficients must be zero.  Hence $\uplambda_2=0$.  This in turn implies that:
\begin{itemize}
\item $\uplambda_1\neq 0$. This is since $a$ belongs to the Jacobson radical but not the Jacobson radical squared, hence so does $\uppsi(a)$.  As $\uplambda_2=0$ above, necessarily $\uplambda_1\neq 0$.
\item $\uplambda_5=0$. This is a consequence of the coefficient $\uplambda_1^2\uplambda_5 + \uplambda_2^2\uplambda_3 + \uplambda_2\uplambda_5^2$ being zero, together with the fact that $\uplambda_2=0$ and $\uplambda_1\neq 0$. 
\end{itemize}
Then, observing that $\uppsi(ab+ba)=0$ since $ab+ba=0$, we see that
\[
\uppsi(a)\uppsi(b)+\uppsi(b)\uppsi(a)=0.
\]
Again, multiplying out the above expressions (using $\uplambda_1=\uplambda_5=0$, see e.g.\ \S\ref{app Acon}), expressing in terms of the basis of $\Lambda_{\con}$ gives
\[
2(\uplambda_1\upmu_1)y^2 +  
2(\uplambda_3\upmu_2)y^2x +
2(\uplambda_1\upmu_5)yx^2 +
 2(\uplambda_1\upmu_7 + \uplambda_3\upmu_5 -
    \uplambda_4\upmu_4 + \uplambda_6\upmu_2 + \uplambda_7\upmu_1)y^2x^2=0.
\]
Since $\uplambda_1\neq 0$, necessarily $\upmu_1=\upmu_5=0$.  Again, since $b$ belongs to the radical but not the radical squared, $\upmu_2\neq 0$.

Finally, since $-\uppsi(a)^2+\uppsi(b)^3+\uppsi(a)\uppsi(b)\uppsi(a)=0$, multiplying out and expressing in terms of the basis of $\Lambda_{\con}$ (again see e.g.\ \S\ref{app Acon}), using $\uplambda_2=\uplambda_5=\upmu_1=\upmu_5=0$, we see
\[
(-\uplambda_1^2 + \upmu_2^3)y^2+(-\uplambda_1^2\upmu_2)y^2x +(-2\uplambda_1\uplambda_7 + \uplambda_4^2 +
    3\upmu_2^2\upmu_3)y^2x^2=0
\]
Hence $\uplambda_1^2\upmu_2=0$, which is a contradiction.  Thus the isomorphism $\uppsi$ cannot exist.
\end{proof}

\begin{remark}\label{intro ok}
Although we don't strictly need this to show \ref{Acon not iso}, in the notation of the introduction it turns out, e.g. using the Shirayanagi algorithm \cite{Shirayanagi}, that
\[
\Gamma_{\con}\cong \frac{\widehat{\mathbb{A}}_{-1}}{x^4+x^3-y^2}.
\]
Thus $\Gamma_{\con}$ is a Jacobi algebra, given by the superpotential $W=x^5+x^4-xy^2$.  Combining this fact with \ref{Acon not iso} justifies the non-isomorphism \eqref{2 cont alg} in the introduction.
\end{remark}

We next show that commutative deformations cannot determine flopping neighbourhoods. This requires the following.
\begin{prop}\label{ab are iso}
$\Lambda_{\con}^{\ab}\cong\Gamma_{\con}^{\ab}$.
\end{prop}
\begin{proof}
By simply commuting variables in the presentations from \ref{L cont alg} and \ref{R cont alg}, 
\[
\Lambda_{\con}^{\ab}\cong\frac{\mathbb{C}[x,y]}{xy,x^2+y^3}
\quad\mbox{and}\quad
\Gamma_{\con}^{\ab}\cong\frac{\mathbb{C}[a,b]}{ab,a^2(b-1)+b^3}
=
\frac{\mathbb{C}[a,b]}{ab,-a^2+b^3}
\]
where the last equality holds simply since $ab=0$ implies $a^2b=0$.  The above two rings are visibly isomorphic.
\end{proof}

\begin{cor}
Flopping neighbourhoods are not determined by the commutative deformations of the reduced flopping curve.
\end{cor}
\begin{proof}
The commutative deformations of $\cO_{\mathbb{P}^1}(-1)$ are given by the abelianization of the contraction algebra \cite[3.2]{DW1}.   Since $\Lambda_{\con}^{\ab}\cong\Gamma_{\con}^{\ab}$ by \ref{ab are iso}, the commutative deformations of the reduced flopping curves in both flopping contractions  are the same.  However, the flops are not analytically isomorphic by \ref{not iso} or \ref{Acon not iso}.
\end{proof}

\appendix
\section{Code for Verification}\label{theappendix}

In this appendix we list computer algebra code which can be used to independently verify the claims made in the main text.

\subsection{Code for the Flops}\label{app flops}
\begin{itemize}
\item $R$ is has a unique singular point at the origin. Singular \cite{DGPS}:\\[1mm]
$
\begin{array}{l}
\tt{LIB``homolog.lib";}\\
\tt{ring\,\,  r=0,(u,v,x,y),dp;}\\
\tt{ideal\,\, i=u2+v2*(x+y)+x*(x2+xy2+y3);}\\
\tt{minAssGTZ(radical(slocus(std(i))));}\\
\end{array}
$
\smallskip

\item $I\colonequals (vx-uy, xy^2+v^2, x^2y+uv)$ is a rank one reflexive $R$-module. Using Macaulay2 \cite{M2}:\\[1mm]
$
\begin{array}{l}
\tt{loadPackage ``Divisor";}\\
\tt{R = QQ[u,v,x,y]/ideal(u^2+v^2*(x+y)+x*(x^2+x*y^2+y^3));}\\
\tt{i=ideal(v*x-u*y, x*y^2+v^2, x^2*y+u*v);}\\
\tt{isReflexive(i)};
\end{array}
$
\smallskip

\noindent
Similarly $(x^2+y^3,vx+uy,ux-vy^2)$ is a rank one reflexive $L$-module.
\smallskip
\item Blowup of $\Spec R$ at the ideal I. Singular:\\[1mm]
$
\begin{array}{l}
\tt{LIB``homolog.lib";}\\
\tt{LIB``resolve.lib";}\\
\tt{ring\,\, r=0,(u,v,x,y),dp;}\\
\tt{ideal\,\, i=u2+v2*(x+y)+x*(x2+xy2+y3);}\\
\tt{ideal\,\, Z=vx-uy, xy2+v2, x2y+uv;}\\
\tt{list\,\, blow=blowUp(i,Z);}\\
\tt{blow;}\\
\end{array}
$\\
$
\begin{array}{l}
\hspace{-0.65em}\left.\begin{array}{l}
\tt{def\,\, Q=blow[1];}\\
\tt{setring\,\, Q;}\\
\tt{sT;}\\
\tt{bM;}\\
\tt{dim\textunderscore slocus(sT);}\\
\tt{elim1(sT,x(1));}\\
\end{array}\right\} 
\mbox{Chart 1}\\
\end{array}
\qquad \begin{array}{l}
\left.\begin{array}{l}
\tt{def\,\, Q=blow[2];}\\
\tt{setring\,\, Q;}\\
\tt{sT;}\\
\tt{bM;}\\
\tt{dim\textunderscore slocus(sT);}
\end{array}\right\} 
\mbox{Chart 2}
\end{array}
$
\smallskip

\noindent
This returns the defining equations of each chart, the map to the base, and also the fact that each chart is smooth.  
\smallskip

\item Milnor and Tjurina numbers of $R$ are $12$ and $10$ respectively.  Singular: \\
$
\begin{array}{l}
\tt{LIB ``sing.lib";}\\
\tt{ring \,\, r=0,(u,v,x,y),dp;}\\
\tt{poly\,\, f=u2+v2*(x+y)+x*(x2+xy2+y3);}\\
\tt{milnor(f);tjurina(f);}\\
\end{array}
$
\smallskip

\noindent
Complete locally the Milnor number of $R$ drops to $11$, whilst the Tjurina number is still $10$.  Singular:\\[1mm]
$
\begin{array}{l}
\tt{LIB ``sing.lib";}\\
\tt{ring \,\, r=0,(u,v,x,y),ds;}\\
\tt{poly\,\, f=u2+v2*(x+y)+x*(x2+xy2+y3);}\\
\tt{milnor(f);tjurina(f);}\\
\end{array}
$
\smallskip

\item The Milnor and Tjurina numbers of $L$ are coded similarly, and are both $11$ in both dp and ds ordering.
\end{itemize}

\subsection{Code for the Contraction Algebras}\label{app Acon}

\begin{itemize}
\item For the new flop $R$, using  \ref{CA summary}\eqref{CA summary 3} the dimension of $\Gamma_{\con}$ is nine. Singular: \\[1mm]
$
\begin{array}{l}
\tt{LIB``homolog.lib";}\\
\tt{ring\,\, r= 0,(u,v,x,y),dp;}\\
\tt{ideal\,\, i=u2+v2*(x+y)+x*(x2+xy2+y3);}\\
\tt{qring\,\, S=std(i);}\\
\tt{module\,\, Ma=[vx-uy], [xy2+v2], [x2y+uv];}\\
\tt{module\,\, M=syz(Ma);}\\
\tt{module\,\, X=prune(syz(M));}\\
\tt{depth(X);}\tt{vdim(Ext(1,X,X));}\tt{vdim(Ext(2,X,X));}\\
\end{array}
$
\smallskip

\noindent
This can be easily adapted, using the reflexive module $(x^2+y^3,vx+uy,ux-vy^2)$ of $L$, to show that  $\Lambda_{\con}$ also has dimension nine.
\smallskip

\item The dimension and basis of $\Gamma_{\con}$.  The code for $\Lambda_{\con}$ is similar. Magma \cite{magma}:\\[1mm]
$
\begin{array}{l}
\tt{K := Rationals();}\\
\tt{F<a,b> := FreeAlgebra(K,2);}\\
\tt{I:=ideal<F\mid a*b + b*a , -a^2 + b^3 + a*b*a>;}\\
\tt{R:=F/I;}\\
\tt{W, Wf :=VectorSpace(R);}\\
\tt{[W.i@@Wf\colon  i\,\, in\,\, [1 .. Dimension(W)]];}\\
\end{array}
$
\smallskip

\item Expressing products in terms of the basis in the proof of \ref{Acon not iso}. Magma:\\[1mm]
$
\begin{array}{l}
\tt{K := Rationals();}\\
\tt{k3<l_1,l_2,l_3,l_4,l_5,l_6,l_7,l_8,m_1,m_2,m_3,m_4,m_5,m_6,m_7,m_8> := RationalFunctionField(K,16);}\\
\tt{F<x,y> := FreeAlgebra(k3,2);}\\
\tt{I:=ideal<F\mid x^3-y^2,  x*y+y*x>;}\\
\tt{T:=F/I;}\\
\tt{W, Wf :=VectorSpace(T);}\\
\tt{A:= l_1*y + l_2*x  + l_3*y^2+ l_4*y*x + l_5*x^2 + l_6*y^2*x + l_7*y*x^2 + l_8*y^2*x^2;}\\
\tt{Wf(A^3);}
\end{array}
$
\smallskip

\noindent
Output shows that $l_2=l_5=0$. Setting these to be zero: \\[1mm]
$
\begin{array}{l}
\tt{K := Rationals();}\\
\tt{k3<l_1,l_3,l_4,l_6,l_7,l_8,m_1,m_2,m_3,m_4,m_5,m_6,m_7,m_8> := RationalFunctionField(K,14);}\\
\tt{F<x,y> := FreeAlgebra(k3,2);}\\
\tt{I:=ideal<F\mid x^3-y^2,  x*y+y*x>;}\\
\tt{T:=F/I;}\\
\tt{W, Wf :=VectorSpace(T);}\\
\tt{A:= l_1*y  + l_3*y^2+ l_4*y*x  + l_6*y^2*x + l_7*y*x^2 + l_8*y^2*x^2;}\\
\tt{B:= m_1*y + m_2*x  + m_3*y^2+ m_4*y*x + m_5*x^2 + m_6*y^2*x + m_7*y*x^2 + m_8*y^2*x^2;}\\
\tt{Wf(A*B+B*A);}
\end{array}
$
\smallskip

\noindent
Output shows that $m_1=m_5=0$.  Setting these to be zero: \\[1mm]
$
\begin{array}{l}
\tt{K := Rationals();}\\
\tt{k3<l_1,l_3,l_4,l_6,l_7,l_8,m_2,m_3,m_4,m_6,m_7,m_8> := RationalFunctionField(K,12);}\\
\tt{F<x,y> := FreeAlgebra(k3,2);}\\
\tt{I:=ideal<F\mid x^3-y^2,  x*y+y*x>;}\\
\tt{T:=F/I;}\\
\tt{W, Wf :=VectorSpace(T);}\\
\tt{A:= l_1*y  + l_3*y^2+ l_4*y*x  + l_6*y^2*x + l_7*y*x^2 + l_8*y^2*x^2;}\\
\tt{B:= m_2*x  + m_3*y^2+ m_4*y*x + m_6*y^2*x + m_7*y*x^2 + m_8*y^2*x^2;}\\
\tt{Wf(-A^2+B^3+A*B*A);}
\end{array}
$
\smallskip

\noindent
Output shows that $l_1^2m_2=0$.
\end{itemize}

\end{document}